\documentclass[11pt]{amsart}
\usepackage{amsmath,amssymb,latexsym,esint,cite,mathrsfs}
\usepackage{verbatim,wasysym,mathrsfs}
\usepackage[left=2.8cm,right=2.8cm,top=2.86cm,bottom=2.86cm]{geometry}

\usepackage{microtype}
\usepackage{color,enumitem,graphicx}
\usepackage[colorlinks=true,urlcolor=blue, citecolor=red,linkcolor=blue,
linktocpage,pdfpagelabels,bookmarksnumbered,bookmarksopen]{hyperref}
\usepackage[hyperpageref]{backref}
\usepackage[english]{babel}
\renewcommand{\epsilon}{\varepsilon}
\renewcommand{\H}{{\mathcal H}}
\newcommand{\pnorm}[2][]{\if #1'' \left|#2\right|_p \else \left|#2\right|_{#1} \fi}

\begin{document}

\title[Fractional $p$-Laplacian problems with singular weights]{A note on fractional $p$-Laplacian \\ problems with singular weights}

\author[K.\ Ho, K.\ Perera, I.\ Sim and M.\ Squassina]{Ky Ho, Kanishka Perera, Inbo Sim, and Marco Squassina}

\address[K.\ Ho]{NTIS \newline
	University of West Bohemia, Technick\'a 8, 306 14 Plze\v{n}, Czech Republic}
\email{ngockyh@ntis.zcu.cz}

\address[K.\ Perera]{
Department of Mathematical Sciences \newline
Florida Institute of Technology, Melbourne, FL 32901, USA}
\email{kperera@ﬁt.edu}

\address[I.\ Sim]{
Department of Mathematics \newline
University of Ulsan, Ulsan 680-749, Republic of Korea}
\email{ibsim@ulsan.ac.kr}

\address[M.\ Squassina]{Dipartimento di Informatica \newline
University of Verona, Strada Le Grazie I-37134 Verona, Italy}
\email{marco.squassina@univr.it}

\thanks{K.\ Ho was supported by the project LO1506 of the Czech Ministry of Education, Youth and Sports.
	I.\ Sim was supported by NRF Grant No.\ 2015R1D1A3A01019789.\
	M.\ Squassina is member of Gruppo Nazionale per l'Analisi Matematica, la Probabilit\`a
	e le loro Applicazioni (GNAMPA)}

\begin{abstract}
We study a class of fractional $p$-Laplacian problems with weights which are possibly singular on the boundary of the domain.
We provide existence and multiplicity results as well as characterizations of critical groups and related applications.
\end{abstract}

\dedicatory{Dedicated with admiration to Paul Rabinowitz, a master in Nonlinear Analysis}

\maketitle

%\begin{center}
%	\begin{minipage}{9cm}
%		\small
%		\tableofcontents
%	\end{minipage}
%\end{center}
%
%\medskip

\numberwithin{equation}{section}
\newtheorem{theorem}{Theorem}[section]
\newtheorem{lemma}[theorem]{Lemma}
\newtheorem{proposition}[theorem]{Proposition}
\newtheorem{corollary}[theorem]{Corollary}
\newtheorem{definition}[theorem]{Definition}
\newtheorem{example}[theorem]{Example}
\newtheorem{remark}[theorem]{Remark}
\allowdisplaybreaks

\newcommand{\A}{{\mathscr A}}
\newcommand{\B}{{\mathscr B}}
	
%################## SECTION 1. INTRODUCTION #################### %
\section{Introduction}
Let $s \in (0,1)$ and $\Omega\subset\mathbb{R}^{N}$ be a bounded domain with Lipschitz boundary $\partial \Omega$. Consider also
a Carath\'{e}odory function $f:\Omega\times\mathbb{R}\to\mathbb{R}$. Recently, the following semi-linear problem involving
the fractional Laplacian has been the subject of various investigations
\begin{eqnarray*}
	%\label{1.1}
\begin{cases}
 (-\Delta)^su=f(x,u) \quad &\text{in } \Omega,\\
  u=0\quad &\text{in } \mathbb{R}^N\setminus \Omega.
  \end{cases}
\end{eqnarray*}
The nonlocal operator $(-\Delta)^s$ naturally arises in various fields, such as continuum mechanics, phase transition phenomena,
population dynamics, game theory and financial mathematics  \cite{A,C}.
Existence \cite{SV, SV3}, non-existence \cite{RS1} and regularity \cite{CS1, RS} have been studied intensively.
In this paper, we consider quasi-linear problems
\begin{eqnarray}\label{1.2}
\begin{cases}
(-\Delta)_p^su=f(x,u) \quad &\text{in } \Omega,\\
u=0\quad &\text{in } \mathbb{R}^N\setminus \Omega,
\end{cases}
\end{eqnarray}
containing a nonlocal nonlinear operator known as the fractional $p$-Laplacian,
which represents a natural extension of fractional Laplacian. For $p \in (1,\infty)$ and $u$ smooth enough,
\begin{equation*}
	%\label{1.3}
(- \Delta)_p^s\, u(x) = 2\, \lim_{\varepsilon \searrow 0} \int_{\mathbb{R}^N \setminus B_\varepsilon(x)} \frac{|u(x) - u(y)|^{p-2}\, (u(x) - u(y))}{|x - y|^{N+sp}}\, dy, \quad x \in \mathbb{R}^N.
\end{equation*}
We refer to \cite{C} for the motivations that lead to
the introduction of this operator.\ So far, existence and regularity for the above problem have been investigated \cite{BP,erk-lind,kasm,IMS1, IS, MRmpsy} under the assumption that the function $f$ is $L^\infty$-Carath\'{e}odory, namely {\em nonsingular} in the $x$-dependence on $\partial\Omega$. Our goal in this paper is to get existence results for \eqref{1.2} when $f$ involves {\em singular weights}. It is worth noticing that this is new even for the plain fractional Laplacian.
This paper is motivated by \cite{PS} where the $p$-Laplace equation $- \Delta_p u= f(x,u)$ was investigated with singular weights, namely
$\Delta_p u = \mbox{div} (|\nabla u|^{p-2} \nabla u)$ and $f$ satisfies the subcritical growth condition
%\begin{equation} \label{1.5}
%|f(x,t)| \le \sum_{i=1}^n h_i(x)\, |t|^{q_i - 1} \quad \text{for a.e.\ $x \in \Omega$ and all $t \in \mathbb{R}$,}
%\end{equation}
\begin{equation} \label{1.5}
|f(x,t)| \le h_1(x)\, |t|^{q_1 - 1}+\cdots +h_n(x)\, |t|^{q_n - 1}\quad \text{for a.a.\ $x \in \Omega$ and all $t \in \mathbb{R}$,}
\end{equation}
for some $q_i \in [1,p^\ast)$ with $p^\ast := Np/(N-p)$ and measurable weights $h_i \ge 0$ which are 
possibly singular along the boundary $\partial \Omega$. Admissible classes of weights are introduced in \cite{PS}, which 
are appropriate in order to use H\"older and Hardy inequalities to show that the
functional associated with \eqref{1.2} is well-defined and the Palais-Smale condition
holds at any energy level, allowing to obtain several existence results. Previously, the semi-linear
case was considered e.g.\ in \cite{Ry,otani1,otani2,usami}.

%We observe that, for the case of the fractional $p$-Laplacian, it is impossible to have the Hardy inequality 
%holding true for all $s \in (0,1)$ since, in fact, it is true only when $s\neq 1/p$ (see \cite{Dyda}).\
By introducing a suitable class of weights we will get results about existence, multiplicity and characterization of critical groups for any $s\in (0,1).$
%$\mathcal{A}_q$ and $\widetilde{\mathcal{A}}_q$ (cf.\ Definition \ref{classC} and \ref{C-sec})
More precisely, in Theorem~\ref{mult}, we get a multiplicity result for $f(x,u)=h(x)|u|^{q-2}u$ with $q\neq p$.
Theorems~\ref{Theorem 4.1}, \ref{Theorem 4.2} and \ref{Theorem 4.6} are
about the computation of critical groups at zero of the functional associated with \eqref{1.2} when $f(x,u)$ is a sum of terms
with singular weight enjoying proper summability.
Moreover, in Proposition~\ref{bound} we prove the boundedness of solutions for
$f(x,u) = \lambda h(x)|u|^{p-2}u$ with $h$ belonging to a suitable class.
In Theorem~\ref{nontrivialth}, nontrivial solutions are found
under various conditions on $f$.
Finally, in Theorem~\ref{multip}, we establish the existence of infinitely many solutions when $f(x,u)$ is odd in $u$.

\vskip4pt
The paper is organized as follows. In Section 2, we provide a suitable functional framework 
for problem \eqref{1.2} and prove some preliminary results. 
In Section 3, we consider related eigenvalue problems. 
In Section 4, we compute the critical groups of the functional associated with \eqref{1.2}.
Section 5 is devoted to show existence of nontrivial solutions via cohomological local splitting 
and critical groups. Finally, in Section 6, we obtain the existence of infinitely many solutions.

%############################ SECTION 2. PRELIMINARIES#################### %

\section{Functional framework and preliminaries}

\noindent
Throughout the paper we will assume that 
$\Omega\subset {\mathbb R}^N$ is a bounded Lipschitz domain with $N\geq 2$.
\vskip1pt
\noindent
In this section, we provide the variational setting on a suitable function space for \eqref{1.2}, jointly with some preliminary results.
We consider, for any $p\in (1,\infty)$ and $s\in (0,1)$, the space
$$
W^{s,p}_{0}(\Omega):=\big\{u\in W^{s,p}(\mathbb{R}^N):\, u=0 \,\ \text{in}\ \mathbb{R}^N\setminus \Omega \big\},
$$
endowed with the standard Gagliardo norm
\begin{equation}
\label{norm}
\|u\|:=\left(\int_{{\mathbb R}^{2N}}\frac{|u(x)-u(y)|^p}{|x-y|^{N+sp}}dxdy\right)^{1/p}.
\end{equation}
We observe that, as it can be readily seen, this norm is equivalent to the full norm
\begin{equation*}
u\mapsto  \left(\int_{{\mathbb R}^{N}}|u(x)|^pdx+\int_{{\mathbb R}^{2N}}\frac{|u(x)-u(y)|^p}{|x-y|^{N+sp}}dxdy\right)^{1/p}\!\!,
\end{equation*}
namely a Poincar\'e inequality holds in $W^{s,p}_{0}(\Omega)$. Let
$$
p_s^\ast :=\frac{Np}{N-sp},
$$
with the agreement that $p_s^\ast =\infty$ if $N\leq sp$.
It is well-known that $W^{s,p}_{0}(\Omega)$ is a uniformly convex reflexive
Banach space, continuously embedded into $L^q(\Omega)$ for all $q\in [1,p^\ast_s]$
if $N>sp$, for all $1\leq q<\infty$ if $N=sp$ and into $L^\infty(\Omega)$ for $N<sp$. It is also compactly injected in $L^q(\Omega)$ for any $q\in [1,p^\ast_s)$ if $N\geq sp$ and into $L^\infty(\Omega)$ for $N<sp$.
Furthermore, $C_0^{\infty }(\Omega)$ is a dense subspace of $W^{s,p}_{0}(\Omega)$ with respect to the norm \eqref{norm}.\
In particular, restrictions to $\Omega$ of functions in $W^{s,p}_{0}(\Omega)$
belong to the closure of $C_0^{\infty }(\Omega)$ in $W^{s,p}(\Omega)$, i.e.\ with respect to the localized norm
$$
\|u\|_{W^{s,p}_{0}(\Omega)}:=\left(\int_{\Omega}|u(x)|^pdx+\int_{\Omega}\int_{\Omega}\frac{|u(x)-u(y)|^p}{|x-y|^{N+sp}}dxdy\right)^{1/p}.
$$
This closure is often denoted with the same symbol $W^{s,p}_{0}(\Omega)$.\ Notice that
for the seminorm localized on $\Omega\times\Omega$ there is no Poincar\'e inequality with $\int_{\Omega}|u|^pdx$ if $sp\leq 1$, cf.\
\cite[Remark 2.4]{BP}.
%From \cite[Corollary 1.4.4.5]{Grisvard} we learn that
%if $\Omega\subset\mathbb{R}^N$ is a Lipschitz bounded domain and $sp\ne 1$ then
%$W^{s,p}_0 (\Omega) = \widetilde{W}^{s,p}(\Omega ),$ where $\widetilde{W}^{s,p}(\Omega )$ is the space of $u\in W^{s,p}(\Omega)$ such that
%$\widetilde{u}\in W^{s,p}(\mathbb{R}^N),$ where $\widetilde{u}$ denotes the extension
%by zero of $u$ out of $\Omega$. Also, if $sp < 1$, then $W^{s,p}_0 (\Omega) = W^{s,p}(\Omega ).$
\vskip4pt
\noindent
Next we state some Hardy-type inequalities from \cite[Theorem 1.1 and formula (17)]{Dyda}, for the
cases $sp\neq 1$ and \cite[Theorem 6.5]{hardyy}, for the particular case $sp=1$.

\begin{theorem}[Hardy inequality]
	\label{Hardy}
%	Let $\Omega\subset {\mathbb R}^N$ be a bounded Lipschitz domain with $N\geq 2$.
For any $p\in (1,\infty)$ and $s\in (0,1)$ the following facts hold:
\begin{itemize}
\item[$\bullet$] if $sp>1$, then for any $u\in W^{s,p}_{0}(\Omega)$ we have
$$
\int_{\Omega}\frac{|u(x)|^p}{\operatorname{dist}(x,\partial\Omega)^{sp}}dx
\leq C\int_{\Omega}\int_{\Omega}\frac{|u(x)-u(y)|^p}{|x-y|^{N+sp}}dxdy,
$$
where $C$ is a positive constant depending only on $\Omega,N,p$ and $s$. 
\vskip6pt
\item[$\bullet$]  if $sp<1$, then for any $u\in W^{s,p}_{0}(\Omega)$ we have
$$
\int_{\Omega}\frac{|u(x)|^p}{\operatorname{dist}(x,\partial\Omega)^{sp}}dx
\leq C'\left(\int_{\Omega}|u|^p\, dx+\int_{\Omega}\int_{\Omega}\frac{|u(x)-u(y)|^p}{|x-y|^{N+sp}}dxdy\right)\!,
$$
where $C'$ is a positive constant depending only on $\Omega,N,p$ and $s$.
\vskip6pt 
\item[$\bullet$]  if $sp=1$, then for any $u\in W^{s,p}_{0}(\Omega)$ we have
$$
\int_{\Omega}\frac{|u(x)|^p}{\operatorname{dist}(x,\partial\Omega)^{sp}}dx
\leq C''\int_{{\mathbb R}^{2N}}\frac{|u(x)-u(y)|^p}{|x-y|^{N+sp}}dxdy,
$$
where $C''$ is a positive constant depending only on $\Omega,N,p$ and $s$. 
\end{itemize}
\end{theorem}

\noindent
As a by-product of the previous result, in any case, by \eqref{norm} and the related Poincar\'e inequality, 
$$
%sp\neq 1\quad\Longrightarrow\quad
\int_{\Omega}\frac{|u(x)|^p}{\operatorname{dist}(x,\partial\Omega)^{sp}}dx
\leq C\|u\|^p, \,\,\quad
\text{for any $u\in W^{s,p}_{0}(\Omega)$,}
$$
where $C$ is a positive constant depending only on $\Omega,N,p$ and $s$. 
Let us denote by
$$
\rho(x):=\operatorname{dist}(x,\partial\Omega), \,\quad x\in \Omega,
$$
the distance from $x\in\Omega$ to $\partial \Omega$ and by $|\cdot|_p$ the usual norm in the space $L^p(\Omega).$

\vskip5pt
\noindent
We consider the following classes of singular weights.

\begin{definition}[Class of weights $\A_q$]
	\label{classA}
For $q\in[1,p_s^\ast)$, let $\A_q$ denote the class of
measurable functions $h$ such that $h\in L^r(\Omega)$ for some $r\in (1,\infty)$ satisfying
$\frac{1}{r}+\frac{q}{p_s^\ast}<1.$
\end{definition}

\begin{definition}[Class of weights $\B_q$]
	\label{classB}
For $q\in[1,p_s^\ast)$, let $\B_q$ denote the class
of measurable functions $h$ such that $h\rho^{sa}\in L^r(\Omega)$ for some $a\in [0,q-1]$ and $r\in (1,\infty)$ satisfying
$\frac{1}{r}+\frac{a}{p}+\frac{q-a}{p_s^\ast}<1.$
\end{definition}

%\begin{definition}[Unifying class of weights]
%	\label{classC}
%For $q\in[1,p_s^\ast)$, we say that $h\in\mathcal{C}_q$ if
%$$
%\text{$h\in\mathcal{A}_q$ when $sp=1$,\qquad
%$h\in\mathcal{B}_q$ when $sp\neq 1$.}
%$$
%\end{definition}

\noindent
Clearly, $\A_q\subset\B_q,$ by simply choosing $a=0$.
Explicitly, $h$ belongs to the above classes provided that there exist $r>1$
and  $0\leq a\leq q-1$ with
\begin{align*}
 \int_{\Omega} |h(x)|^rdx&<\infty,\quad\text{with\,\, $\frac{1}{r}+\frac{q}{p_s^\ast}<1$, \quad\,\, ($h\in \A_q$),} \\
 \int_{\Omega} |h(x)|^r\rho(x)^{sar}dx&<\infty,\quad\text{with\,\, $\frac{1}{r}+\frac{a}{p}+\frac{q-a}{p_s^\ast}<1$, \quad\,\, ($h\in\B_q$).}
\end{align*}
If, for instance, we consider
$$
h(x)=(1-|x|)^{-\beta},\quad\,\, \Omega=B(0,1),
$$
then $h\in \B_q$ if $\beta<sa+r^{-1}$ for some $r>1$ and  $0\leq a\leq q-1$ with $1/r+a/p+(q-a)/p^*_s<1$.
\vskip4pt
\noindent
The following lemma will be used frequently.

\begin{lemma}
	\label{C}
Let $h\in\B_q$.\ Then there holds
$$
\int_{\Omega}|h(x)||u|^{q-1}|v|dx\leq C\|u\|^{q-1}|v|_b, 
\,\,\quad \text{for every $u,v\in W_0^{s,p}(\Omega)$},
$$
where $b\in(1,p_s^\ast)$ is such that $\frac{1}{r}+\frac{a}{p}+\frac{q-a}{b}=1$.
%\begin{itemize}
%\item[${\rm i})$] $\frac{1}{r}+\frac{q}{b}=1$,\,\, in the case $h\in\A_q$;
%\item[${\rm ii})$] $\frac{1}{r}+\frac{a}{p}+\frac{q-a}{b}=1$,\,\, in the case $h\in\B_q$.
%\end{itemize}
\end{lemma}
\begin{proof}
%If $h\in\A_q$, we have
%\begin{equation*}
%\int_{\Omega}|h(x)||u|^{q-1}|v|dx\leq|h|_r\ |u|^{q-1}_b\ |v|_b.
%\end{equation*}
%Since $W^{s,p}_0(\Omega )\hookrightarrow L^b(\Omega),$ we get the conclusion.
If $h\in\B_q$, we have
\begin{equation*}
\int_{\Omega}|h(x)||u|^{q-1}|v|dx=\int_{\Omega}|h\rho^{sa}|\left|\frac{u}{\rho^s}\right|^{a}|u|^{q-1-a}|v|dx
\leq|h\rho^{sa}|_r\ \left|\frac{u}{\rho^s}\right|_p^{a}\ |u|_b^{q-1-a}\ |v|_b.
\end{equation*}
In light of Theorem~\ref{Hardy} and $W_0^{s,p}(\Omega)\hookrightarrow L^b(\Omega)$ we get the conclusion.
\end{proof}

\vskip3pt
\noindent
%Let $X=\{u\in W^{s,p}(\mathbb{R}^N):u=0 \ \text{in}\ \mathbb{R}^N\setminus \Omega \}$ be the uniformly convex Banach space endowed with
%$$
%\|u\|:=\left(\int_{\mathbb{R}^{2N}}\frac{|u(x)-u(y)|^p}{|x-y|^{N+sp}}dxdy\right)^{1/p}.
%$$
%Of course $\|u\|_{W^{s,p}_{0}(\Omega)}\leq C\|u\|$ for some $C>0$ and all $u\in X$.
We now define the operator $A: W_0^{s,p}(\Omega)\to W_0^{-s,p'}(\Omega)$ as
$$
\langle A(u),v\rangle:=\int_{\mathbb{R}^{2N}}\frac{|u(x)-u(y)|^{p-2}(u(x)-u(y))(v(x)-v(y))}{|x-y|^{N+sp}}dxdy,\quad 
\text{for all $u,v\in W_0^{s,p}(\Omega)$.}
$$
A weak solution of problem \eqref{1.2} is a function $u\in W_0^{s,p}(\Omega)$ such that
$$
\langle A(u),v\rangle=\int_{\Omega}f(x,u)vdx,\quad \text{for all $u,v\in W_0^{s,p}(\Omega)$.}
$$
It is easy to see that $A$ satisfies the following compactness condition \cite{IS,MR2640827}:
\vskip4pt
\begin{itemize}
\item[(${\mathcal S}$)] If $\{u_n\}_{n\in{\mathbb N}}\subset W_0^{s,p}(\Omega)$ is such that
$u_n\rightharpoonup u$ in $W_0^{s,p}(\Omega)$  and $\langle A(u_n),u_n-u\rangle\to 0$ as $n\to\infty$,
then $u_n\to u$ in $W_0^{s,p}(\Omega)$ as $n\to\infty$.
\end{itemize}
\vskip1pt
\noindent
Assuming that $f$ satisfies condition \eqref{1.5} for some exponents $q_i\in[1,p_s^\ast)$ and $h_i\in\B_{q_i}$, in light of Lemma~\ref{C}, 
there exists a constant $C > 0$ such that, for all $u, v \in W_0^{s,p}(\Omega)$,
$$
\left|\int_{\Omega} f(x,u)v\, dx\right|\leq \sum_{i=1}^n\int_{\Omega} h_i(x)|u|^{q_i-1}|v|dx\leq
C\sum_{i=1}^n\|u\|^{q_i-1}|v|_{b_i}\leq C_u\|v\|,
$$
so that $f(x,u)\in W_0^{-s,p'}(\Omega)$.\ Weak solutions of \eqref{1.2} are thus critical points of $\Phi:W_0^{s,p}(\Omega)\to{\mathbb R}$,
 $$
 \Phi(u):=\frac{1}{p}\|u\|^p-\int_{\Omega}F(x,u)dx,\qquad F(x,t):=\int_{0}^{t}f(x,s)ds.
 $$
By the property $({\mathcal S})$, we easily obtain the next lemma.

\begin{lemma}[Palais-Smale condition]
	\label{PS1.5}
Assume that $f$ satisfies \eqref{1.5}  for some $q_i\in[1,p_s^\ast)$ and $h_i\in\B_{q_i}.$ Then any
bounded sequence $\{u_n\}_{n\in {\mathbb N}} \subset W_0^{s,p}(\Omega)$ such that $\Phi'(u_n)\to 0$ has a convergent subsequence.\ In particular, bounded Palais-Smale sequences of $\Phi$ are precompact in $W_0^{s,p}(\Omega)$.
\end{lemma}
\begin{proof}
Since $\{u_n\}_{n\in {\mathbb N}}$ is bounded in $W_0^{s,p}(\Omega)$, up to a subsequence, $u_n\rightharpoonup u$ in $W_0^{s,p}(\Omega)$ and
$|u_n-u|_{b_i}\to 0$ as $n\to\infty$ for $i=1,\dots,n$, since $b_i\in(1,p_s^\ast)$.\ Then, in light of Lemma~\ref{C}, we get
$$
\left|\int_{\Omega}f(x,u_n)(u_n-u)dx\right|\leq \sum_{i=1}^n\int_{\Omega}|h_i(x)||u_n|^{q_i-1}|u_n-u|dx\leq C\sum_{i=1}^n\|u_n\|^{q_i-1}|u_n-u|_{b_i},
$$
for some positive constant $C$. This yields
$$
\lim_{n\to\infty}\int_{\Omega}f(x,u_n)(u_n-u)dx=0,
$$
via the boundedness of $\{u_n\}_{n\in {\mathbb N}}$ in $W_0^{s,p}(\Omega)$ and $u_n\to u$ in $L^{b_i}(\Omega)$ for any $i$. Thus,
$$
\langle A(u_n),u_n-u\rangle=\langle \Phi'(u_n), u_n-u\rangle+\int_{\Omega}f(x,u_n)(u_n-u)dx\to 0.
$$
Hence, $u_n\to u$ in $W_0^{s,p}(\Omega)$ as $n\to\infty$, by means of $({\mathcal S})$.
\end{proof}

%################## SECTION 3. AN EIGENLAVUE PROBLEM ################# %
\section{Eigenvalue problems}
We consider the eigenvalue problem
\begin{eqnarray}\label{EP}
\begin{cases}
 (-\Delta)_p^su=\lambda h(x)|u|^{p-2}u \quad &\text{in } \Omega,\\
  u=0\quad &\text{in } \mathbb{R}^N\setminus \Omega,
  \end{cases}
\end{eqnarray}
where $h\in\A_p$ is possibly sign-changing with  $|\{x \in \Omega : h(x) >0\}| >0$ and $\lambda$ is a real number.  We consider the first eigenvalue, defined as follows
$$
\mu_1 := \inf \left\{\frac{1}{p} \|u\|^p : u \in W_0^{s,p}(\Omega),\ \int_\Omega \frac{h(x)}{p} |u|^p\, dx =1\right\}.
$$
We have the following

\begin{theorem} \label{positive}
Let $h\in\B_p$. Then
$\mu_1$ is attained by some nonnegative $\phi_1 \in W_0^{s,p}(\Omega)$. Furthermore, $\phi_1 >0$ a.a.\ if $h\geq 0$,
and any two first eigenfunctions are proportional.
\end{theorem}

\begin{proof}
Since $h\in\B_p$, a standard argument yields the existence of an
eigenfunction $\phi_1 \ge 0$. In fact, notice that, if $\{u_n\}_{n\in{\mathbb N}}
\subset W_0^{s,p}(\Omega)$ is weakly convergent to some $u$, we have
\begin{align*}
\left| \int_{\Omega} h(x)|u_n|^pdx-\int_{\Omega} h(x)|u|^pdx\right|&\leq
C \int_{\Omega} h(x)|u_n|^{p-1}|u_n-u|dx+C\int_{\Omega} h(x)|u_n|^{p-1}|u_n-u|dx \\
&\leq C(\|u_n\|^{p-1}+\|u\|^{p-1})|u_n-u|_b\leq C|u_n-u|_b=o_n(1),
\end{align*}
by Lemma~\ref{C} and since $|u_n-u|_b\to 0$ up to a subsequence, where $b\in(1,p_s^\ast)$ is such that ${1/r}+{a/p}+(p-a)/b=1$, for some $a\in [0,p-1]$.
It follows from \cite[Theorem A.1]{BF} that $\phi_1>0,$ provided $h\geq 0$.
The simplicity follows as in \cite{Franzina}.
%The Harnack inequality of ??? with the boundedness of weak solutions can be applied to conclude the positiveness.
\end{proof}

\noindent
Next we show the
boundedness of weak solutions by modifying the argument in \cite{Franzina}.

%----------------------THEOREM 4.1------------------------------%
\begin{proposition}
	\label{bound}
	Let $u$ be any eigenfunction of \eqref{EP}.\ Then $u\in L^\infty (\mathbb{R}^N)$.
\end{proposition}
%===================== PROOF OF THEOREM 4.1=========================%
\begin{proof}
	The proof follows the line of \cite{Franzina}.
	We shall provide the details for the sake of completeness.
	%Of course, we only need to consider the case $sp\leq N$.
	%$$W^{s,p}(\Omega)\hookrightarrow C_b^{0,s-\frac{N}{p}}(\Omega),$$ when $s>N/p.$
	Denoting $u_+ := \max \{u,0\},$ it suffices to show that, for any weak solution  $u\in W_0^{s,p}(\Omega),$
	\begin{equation}\label{ineq}
	\text{$|u|_\infty\leq 1$ provided that $|u_+|_q\leq \delta$},
	\quad q:=\frac{pr}{r-1}\in (1,p_s^\ast),
	\end{equation}
	for some $\delta>0.$ For each $k \in \mathbb{N}\cup \{0\}$, set $w_k:=(u-(1-1/2^k))_+.$
	Then $w_k\in W_0^{s,p}(\Omega)$ and
	\begin{equation*}
	w_{k+1}(x)\leq w_{k}(x)\quad \text{a.a.},\quad
	u(x)< (2^{k+1}-1)w_{k}(x)\quad \text{a.a. $x\in\{w_{k+1}>0\}$},
	\end{equation*}
	and $\{w_{k+1}>0\}\subseteq \{w_k>2^{-(k+1)}\}$.
	Moreover, for a measurable $v$, we have the inequality
	\begin{equation}\label{ineq.v}
	|v(x)-v(y)|^{p-2}(v_+(x)-v_+(y))(v(x)-v(y))\geq |v_+(x)-v_+(y)|^p
	\end{equation}
	for a.a.\ $x,y\in\mathbb{R}^N.$
	Applying \eqref{ineq.v} for $v=u-(1-1/2^{k+1})$, we have $v_+=w_{k+1}$ and
	\begin{align*}
	\|w_{k+1}\|^p &\leq \int_{\mathbb{R}^{2N}}\frac{|u(x)-u(y)|^{p-2}(u(x)-u(y)(w_{k+1}(x)-w_{k+1}(y))}{|x-y|^{N+sp}}dxdy\\
	&=\lambda\int_{\{w_{k+1}>0\}}h(x)|u(x)|^{p-2}u(x)w_{k+1}(x)dx.
	\end{align*}
	Whence, by the above stated properties, we get
	\begin{equation*}
	\|w_{k+1}\|^p \le|\lambda|(2^{k+1}-1)^{p-1}\int_{\{w_{k+1}>0\}}|h(x)||w_k(x)|^{p}dx.
	\end{equation*}
	By the H\"older inequality, we then obtain
	\begin{equation*}
	\|w_{k+1}\|^p \leq |\lambda|(2^{k+1}-1)^{p-1}|h|_r U_k,\quad  U_k:=|w_k|^p_q.
	\end{equation*}
	Let $\bar{q}$ be such that $q<\bar{q}<p_s^\ast.$ Using again H\"older inequality, we easily get
	\begin{equation*}
	U_{k+1}\leq C\|w_{k+1}\|^p|\{w_{k+1}>0\}|^{\frac{p(\bar{q}-q)}{q\bar{q}}}\label{E2}
	\end{equation*}
	by  the embedding $W_0^{s,p}(\Omega)\hookrightarrow L^{\bar{q}}(\Omega)$.
	On the other hand, Chebychev's inequality entails
	\begin{equation*}
	|\{w_{k+1}>0\}|\leq|\{w_k>2^{-(k+1)}\}|\leq 2^{q(k+1)}U_k^{q/p}.
	\end{equation*}
	Combining the previous inequality yields
	$U_{k+1}\leq C|\lambda|(2^{k+1}-1)^{p-1} |h|_{r}2^{\frac{p(\bar{q}-q)(k+1)}{\bar{q}}}U_k^{1+\frac{\bar{q}-q}{\bar{q}}}$, namely
	\begin{equation*}
	U_{k+1}\leq C_0 b^kU_k^{1+\alpha},\quad \ k\in {\mathbb N},\,\,
	C_0>0,\,\,\alpha>0,\,\, b>1.
	\end{equation*}
	By \cite[Lemma 4.7, Ch. II]{Ladyzhenskaya}, this yields
	$U_k\to 0$ as $k\to\infty$ if $U_0\leq{C_0}^{-\frac{1}{\alpha}} b^{-\frac{1}{\alpha^2}}$.
	Hence, if we choose
	%$$
	%U_0=\left(\int_{\Omega}|u_+|^q dx\right)^{\frac{p}{q}}\leq \widetilde{C}^{-\frac{1}{\alpha}} b^{-\frac{1}{\alpha^2}},$$
	%i.e.,
	$$
	|u_+|_q\leq {C_0}^{-\frac{1}{p\alpha}} b^{-\frac{1}{p\alpha^2}}=:\delta,
	$$
	which is made possible by a simple scaling argument due to the homogeneity of the problem,
	we conclude $U_k\to 0$ as $k\to\infty$, namely $|u|_\infty\leq 1$, via Fatou's Lemma, concluding the proof.
	%$$U_k=\left(\int_{\Omega}(u-(1-1/2^k))_+^q dx\right)^{\frac{p}{q}}\to 0\quad\text{as}\quad k\to\infty,$$
	%i.e.\ $|u|_\infty\leq 1.$ This proves \eqref{ineq} and hence the proof is complete.
\end{proof}

\noindent
Next we consider the eigenvalue problem
\begin{eqnarray}\label{EQ}
\begin{cases}
 (-\Delta)_p^su=\lambda h(x)|u|^{q-2}u \quad &\text{in } \Omega,\\
  u=0\quad &\text{in } \mathbb{R}^N\setminus \Omega,
  \end{cases}
\end{eqnarray}
where $q \in [1,p^\ast_s)$ and $h \in \B_q$ with  $|\{x \in \Omega : h(x) >0\}| >0.$
We shall produce a sequence of eigenvalues for problem \eqref{EQ}, following the argument in \cite{PS}. Let
\[
J(u) := \int_\Omega \frac{h(x)}{q}|u|^q\, dx, \quad u \in W_0^{s,p}(\Omega),
\]
and set (we use the notations of \cite[Chapter 4]{MR2640827})
\[
\Psi(u) := \frac{1}{J(u)}, \quad u \in \mathcal{M},
\qquad  \mathcal{M}:= \left\{u \in W_0^{s,p}(\Omega):\, \frac{1}{p} \|u\|^p = 1 \text{ and } J(u) > 0 \right\}.
\]
Then $\mathcal{M}$ is nonempty and positive eigenvalues and associated eigenfunctions of
\eqref{EQ} on $\mathcal{M}$ coincide with critical values and critical points of $\Psi$, respectively. By Lemma \ref{C}, we have
\[
0 < J(u) \le C \|u\|^q \le C, \quad \text{for all $u \in \mathcal{M}$},
\]
for some constant $C > 0$, and hence $\lambda_1 := \inf_{u \in \mathcal{M}}\, \Psi(u) > 0$. A slight variant of Lemma \ref{PS1.5} yields the following

\begin{lemma}\label{PS1}
For all $c \in \mathbb{R}$, $\Psi$ satisfies the Palais-Smale condition, 
namely every sequence $\{u_n\}_{n\in {\mathbb N}} \subset \mathcal{M}$ with $\Psi(u_n) \to c$ and $\Psi'(u_n) \to 0$ has a subsequence converging to some $u \in \mathcal{M}$.
\end{lemma}

%\begin{proof}
%We have $c \ge \lambda_1$, and there is a sequence $\{\mu_j\} \subset \mathbb{R}$ such that
%\begin{equation} \label{dual}
%\mu_j\, I'(u_j) - \frac{J'(u_j)}{J(u_j)^2} \to 0.
%\end{equation}
%Since $I'(u_j)u_j = p\, I(u_j) = p$, $J'(u_j)u_j = q\, J(u_j)$, and $J(u_j) \to 1/c$, then $\mu_j \to qc/p > 0$. By Lemma \ref{A}-\ref{B},
%\begin{equation} \label{|dual|}
%|J'(u_j)(u_j - u)| \le \int_\Omega |h(x)|\, |u_j|^{q-1}\, |u_j - u|\, dx \le C ||u_j||^{q-1} |u_j - u|_b,
%\end{equation}
%where $b < p^\ast_s$. Since $\{u_j\}$ is bounded in $W^{s,\, p}_0(\Omega)$, a renamed subsequence converges to some $u$ weakly in $W^{s,\, p}_0(\Omega)$ and strongly in $L^b(\Omega)$. Then
%\[
%\int_\Omega |\nabla u_j|^{p-2}\, \nabla u_j \cdot \nabla (u_j - u)\, dx = I'(u_j)(u_j - u) \to 0
%\]
%by \eqref{dual} and \eqref{|dual|}, and hence $u_j \to u$ in $W^{s,\, p}_0(\Omega)$ by the (S)$_+$ property. By continuity, $J(u) = 1/c > 0$ and hence $u \in \mathcal{M}$.
%\end{proof}

Although one can obtain an increasing and unbounded sequence of critical values of $\Psi$ via standard minimax schemes, we
prefer to use a cohomological index as in Perera \cite{P}, which provides additional topological information about the associated critical points.
Let us recall the definition of the $\mathbb{Z}_2$-cohomological index of Fadell and Rabinowitz \cite{MR57:17677}. Let $W$ be a Banach space. For a symmetric subset $M$ of $W \setminus \{0\}$, let $\overline{M} = M/\mathbb{Z}_2$ be the quotient space of $M$ with each $u$ and $-u$ identified, let $f : \overline{M} \to \mathbb{R}\text{P}^\infty$ be the classifying map of $\overline{M}$, and let $f^\ast : H^\ast(\mathbb{R}\text{P}^\infty) \to H^\ast(\overline{M})$ be the induced homomorphism of the Alexander-Spanier cohomology rings. Then the cohomological index of $M$ is defined by
\[
i(M) = \begin{cases}
\sup\, \left\{m \ge 1 : f^\ast(\omega^{m-1}) \ne 0\right\}, & M \ne \emptyset,\\[5pt]
0, & M = \emptyset,
\end{cases}
\]
where $\omega \in H^1(\mathbb{R}\text{P}^\infty)$ is the generator of the polynomial ring $H^\ast(\mathbb{R}\text{P}^\infty) = \mathbb{Z}_2[\omega]$. For example, the classifying map of the unit sphere $S^{m-1}$ in $\mathbb{R}^m$, $m \ge 1$, is the inclusion $\mathbb{R}\text{P}^{m-1} \subset \mathbb{R}\text{P}^\infty$, which induces isomorphisms on $H^q$ for $q \le m-1$, so $i(S^{m-1}) = m$.
Let $\mathcal{F}$ denote the class of symmetric subsets of $\mathcal{M}$, and set
\[
\lambda_k := \inf_{\substack{M \in \mathcal{F}\\[1pt]
i(M) \ge k}}\, \sup_{u \in M}\, \Psi(u), \quad k \ge 1.
\]
Then $\{\lambda_k\}_{k\in {\mathbb N}}$ is a sequence of positive eigenvalues of \eqref{EQ}, $\lambda_k \nearrow + \infty$, and
\begin{equation} \label{index}
i(\left\{u \in \mathcal{M} : \Psi(u) \le \lambda_k \right\}) = i(\left\{u \in \mathcal{M} : \Psi(u) < \lambda_{k+1}\right\}) = k
\end{equation}
if $\lambda_k < \lambda_{k+1}$, see \cite[Propositions 3.52 and 3.53]{MR2640827}.
As a simple application, we consider the problem
\begin{equation} \label{EQwolambda}
\begin{cases}
 (-\Delta)_p^su= h(x)|u|^{q-2}u \quad &\text{in } \Omega,\\
  u=0\quad &\text{in } \mathbb{R}^N\setminus \Omega.
  \end{cases}
\end{equation}

\noindent
By arguing as in the proof of \cite[Theorem 3.2]{PS}, we have the following
\begin{theorem}
	\label{mult}
Let $q \in [1,p^\ast_s)\setminus\{p\}$ and $h \in \B_q$ with $|\{x \in \Omega : h(x) >0\}| >0.$
Then problem \eqref{EQwolambda} admits a sequence of nontrivial weak solutions $\{u_n\}_{n\in {\mathbb N}}\subset W_0^{s,p}(\Omega)$ such that
\begin{enumerate}
\item[${\rm i})$] if $q < p$, then $\|u_n\| \to 0$ as $n\to\infty$,
\item[${\rm ii})$] if $q > p$, then $\|u_n\| \to \infty$ as $n\to\infty$.
\end{enumerate}
\end{theorem}

%\begin{proof}
%See the proof of Theorem 3.2 in \cite{PS}.
%\end{proof}

%################### SECTION 4. CRITICAL GOURPS  ################%

\section{Critical groups}

\noindent
In this section we compute the critical groups at zero of the functional
\[
\Phi(u) = \frac{1}{p} \|u\|^p-\frac{\lambda}{p}\int_\Omega h(x) |u|^pdx-\int_{\Omega} G(x,u)dx, \,\,\quad u \in W_0^{s,p}(\Omega),
\]
where $G(x,t) = \int_0^t g(x,\tau)\, d\tau$, which is associated with the problem
\begin{equation} \label{perturb}
\begin{cases}
 (-\Delta)_p^su= \lambda h(x)|u|^{p-2}u + g(x,u) \quad &\text{in } \Omega,\\
  u=0\quad &\text{in } \mathbb{R}^N\setminus \Omega,
  \end{cases}
\end{equation}
where $\lambda \ge 0$ is a parameter, $h \in \B_p$ with  $|\{x \in \Omega : h(x) >0\}| >0$ and
$g$ is a Carath\'{e}odory function on $\Omega \times \mathbb{R}$ satisfying the subcritical $p$-superlinear growth condition
\begin{equation}\label{ggrowth}
|g(x,t)|\leq\sum_{i=1}^{n}K_i(x)|t|^{q_i-1}\quad \mbox{for a.a.}~ x\in\Omega~\mbox{and all}~ t\in \mathbb{R}
\end{equation}
for some $q_i \in (p,p_s^\ast)$ and $K_i \in \B_{q_i}$.
%where $q_i\in(p,p_s^\ast)$ and $K_i \in \mathcal{C}_{q_i}$ for $i = 1,\dots,n.$
%Problem \eqref{perturb} has the trivial solution $u = 0$, and we study the critical groups of the associated functional
%\[
%\Phi(u) = I(u) - \int_\Omega \left[  \frac{1}{p}\, \lambda\, h(x)\, |u|^p + G(x,u)\right] dx, \quad u \in W^{s,\, p}_0(\Omega),
%\]
%where $G(x,t) = \int_0^t g(x,\tau)\, d\tau$, at $0$.
The critical groups of $\Phi$ at zero are given by
\begin{equation} \label{criticalgroup}
C^q(\Phi,0) := H^q(\Phi^0 \cap U,\Phi^0 \cap U \setminus \{0\}), \quad q \ge 0,
\end{equation}
where $\Phi^0 :=\Phi^{-1}((-\infty,0])$, $U$ is any neighborhood of $0$, and
$H$ denotes the Alexander-Spanier cohomology with $\mathbb{Z}_2$-coefficients.
Following the steps in \cite{PS}, we can obtain

\begin{theorem}[Critical groups I]
	\label{Theorem 4.1}
Assume that $h \in \B_p$ with $|\{x \in \Omega : h(x) >0\}|>0$, $g$ satisfies
\eqref{ggrowth} and $0$ is an isolated critical point of $\Phi$. Then we have
\begin{enumerate}
\item $C^0(\Phi,0) \approx \mathbb{Z}_2$ and $C^q(\Phi,0) = 0$ for $q \ge 1$ in the following cases:
\begin{enumerate}
\item $0 \le \lambda < \lambda_1$;
\item $\lambda = \lambda_1$ and $G(x,t) \le 0$ for a.a. $x \in \Omega$ and all $t \in \mathbb{R}$.
\end{enumerate}
\item $C^k(\Phi,0) \ne 0$ in the following cases:
\begin{enumerate}
\item $\lambda_k < \lambda < \lambda_{k+1}$;
\item $\lambda = \lambda_k < \lambda_{k+1}$ and $G(x,t) \ge 0$ for a.a. $x \in \Omega$ and all $t \in \mathbb{R}$;
\item $\lambda_k < \lambda_{k+1} = \lambda$ and $G(x,t) \le 0$ for a.a. $x \in \Omega$ and all $t \in \mathbb{R}$.
\end{enumerate}
\end{enumerate}
\end{theorem}

\noindent
In the absence of a direct sum decomposition, the main technical tool we use to get an estimate of the critical groups is the notion of a cohomological local splitting introduced in \cite{MR2640827}.
When $sp > N$, it suffices to assume the sign conditions on $G$ in Theorem \ref{Theorem 4.1} for small $|t|$ by the imbedding $W^{s,\, p}_0(\Omega) \hookrightarrow L^\infty(\Omega)$.

%\item[($A2$)] $|f(x,t)|\leq\sum_{i=1}^{n}K_i(x)|t|^{q_i-1}$ for a.a. $x\in\Omega$ and all $t\in \mathbb{R}^N,$ for some $q_i\in(p,\infty)$ and $K_i\in\mathcal{A}_{q_i}.$

%\vspace{0.3cm}

%\item[($B2$)] $|f(x,t)|\leq\sum_{i=1}^{n}K_i(x)|t|^{q_i-1}$ for a.a. $x\in\Omega$ and all $t\in \mathbb{R}^N,$ for some $q_i\in(p,\infty)$ and $K_i\in\mathcal{B}_{q_i},$;

%\end{itemize}

\begin{theorem}[Critical groups II]
	\label{Theorem 4.2}
Assume that $sp > N$, $h \in \B_p$ with $|\{x \in \Omega : h(x) >0\}| >0$, 
$g$ satisfies \eqref{ggrowth} with $q_i>p$ and $0$ is an isolated critical point of $\Phi$.
Then we have
\begin{enumerate}
\item $C^0(\Phi,0) \approx \mathbb{Z}_2$ and $C^q(\Phi,0) = 0$ for $q \ge 1$ if $\lambda = \lambda_1$ and, for some $\delta > 0$, $G(x,t) \le 0$ for a.a. $x \in \Omega$ and $|t| \le \delta$.
\item $C^k(\Phi,0) \ne 0$ in the following cases:
\begin{enumerate}
\item $\lambda = \lambda_k < \lambda_{k+1}$ and, for some $\delta > 0$, $G(x,t) \ge 0$ for a.a. $x \in \Omega$ and $|t| \le \delta$;
\item $\lambda_k < \lambda_{k+1} = \lambda$ and, for some $\delta > 0$, $G(x,t) \le 0$ for a.a. $x \in \Omega$ and $|t| \le \delta$.
\end{enumerate}
\end{enumerate}
\end{theorem}

\noindent
As we will show next, the conclusions of Theorem \ref{Theorem 4.2} hold for $sp \le N$ when the weights $h$ and $K_i$
belong to suitable strengthened subclasses of $\B_p$ and $\B_{q_i}$, respectively.

%\begin{definition}[Class of weights $\widetilde{\mathcal{A}}_q$]
%	\label{Definition 4.3}
%For $sp \le N$ and $q \in [1,p^\ast_s)$, we denote by
%$\widetilde{\mathcal{A}}_q$ the class of functions $K$ with $K \in L^r(\Omega)$
%for some $r \in (1,\infty)$ satisfying $\frac{1}{r} + \frac{q-1}{p^\ast_s} < \frac{sp}{N}.$
%%\begin{equation} \label{subclassA}
%%\frac{1}{r} + \frac{q-1}{p^\ast_s} < \frac{sp}{N}.
%%\end{equation}
%\end{definition}

\begin{definition}[Class of weights $\widetilde{\B}_q$]
	\label{def4-bis}
For $sp \le N$ and $q \in [1,p^\ast_s)$, we denote by $\widetilde{\B}_q$  the class of  functions
$K$ with $K \rho^{sa} \in L^r(\Omega)$ for some $a \in [0,q - 1]$
and $r \in (1,\infty)$ satisfying $$\frac{1}{r} + \frac{a}{p} + \frac{q-1-a}{p^\ast_s} < \frac{sp}{N}.$$
%\begin{equation} \label{subclassB}
%\frac{1}{r} + \frac{a}{p} + \frac{q-1-a}{p^\ast_s} < \frac{sp}{N}.
%\end{equation}
\end{definition}

\begin{remark}\rm
Note that $\widetilde{\B}_q = \B_q$ when $sp = N$ and $\widetilde{\B}_q \subset \B_q$
%$$
%\widetilde{\mathcal{A}}_q\subset \mathcal{A}_q\cap\widetilde{\mathcal{B}}_q,\quad \widetilde{\mathcal{B}}_q \subset \mathcal{B}_q
%$$
when $sp<N$ since $$\frac{1}{p^\ast_s} + \frac{sp}{N} < 1.$$ 
%We do not have $\A_q \subset \widetilde{\B}_q$ when $sp < N$.
\end{remark}

%\begin{definition}[Unifying class of weights]
%	\label{C-sec}
%For $sp \le N$ and $q \in [1,p^\ast_s)$,  we denote by $\widetilde{\mathcal{C}}_q$ the class $\widetilde{\mathcal{A}}_q$ when $sp=1$ and the class $\widetilde{\mathcal{B}}_q$ when $sp\neq 1$.
%\end{definition}

\noindent
We have the following 
\begin{lemma} \label{Lemma 4.4}
Let $sp \le N$, $q \in [1,p^\ast_s)$ and $K \in \widetilde{\B}_q$. Then there exists $\tau > N/sp$ such that
\[
\left|K(x)\, |u|^{q-1}\right|_\tau \le C \|u\|^{q-1}
\]
for all $u \in  W_0^{s,p}(\Omega),$ for some constant $C > 0$.
\end{lemma}

\begin{proof}
%We only prove the assertion for the case $K\in \widetilde{\mathcal{B}}_q$ when $sp\neq 1.$
Let $0\leq a\leq q-1$ and $r>1$ be as in Definition \ref{def4-bis}. Then, there exists
$b < p^\ast_s$ sufficiently close to $p^\ast_s$ such that
\begin{equation*}
\frac{1}{r} + \frac{a}{p} + \frac{q-1-a}{b} < \frac{sp}{N}.
\end{equation*}
Then, choosing
$$
\tau:=\left(\frac{1}{r} + \frac{a}{p} + \frac{q-1-a}{b} \right)^{-1}>\frac{N}{sp},
$$
by the H\"older inequality, we obtain
$$
\int_\Omega |K(x)|^\tau\, |u|^{(q-1)\, \tau}\, dx = \int_\Omega |K \rho^{sa}|^\tau \left|\frac{u}{\rho^s}\right|^{a\tau} |u|^{(q-1-a)\, \tau}\, dx \le
\left|K \rho^{sa}\right|^\tau_r \left|\frac{u}{\rho^s}\right|^{a\tau}_p |u|^{(q-1-a)\, \tau}_b.
$$
%where $\tau/r + a\tau/p + (q-1-a)\, \tau/b = 1$ and hence $\tau > N/sp$ by \eqref{4.5}.
In light of Theorem \ref{Hardy} and $|u|_b \le C \|u\|$, the conclusion follows.
\end{proof}

\noindent
Solutions of $(- \Delta)_p^s u = f(x)$ enjoy the useful $L^q$-estimate given next (cf.\ \cite[Lemma 2.3]{MRmpsy}).

\begin{lemma}[Summability lemma]
	\label{summab}
	Let $f \in L^q(\Omega)$ for some $1 < q \le \infty$ and assume that $u \in W^{s,p}_0(\Omega)$ is a
	weak solution of the equation $(- \Delta)_p^s\, u = f(x)$ in $\Omega$. Then
	\begin{equation} \label{12}
	\pnorm[r]{u} \le C \pnorm[q]{f}^{1/(p-1)},
	\end{equation}
	where
	\[
	r := \begin{cases}
	\dfrac{N\, (p - 1)\, q}{N - spq}, & 1 < q < \dfrac{N}{sp}, \\[10pt]
	\infty, & \dfrac{N}{sp} < q \le \infty,
	\end{cases}
	\]
	and $C = C(N,\Omega,p,s,q) > 0$.
%	In particular, if $f \in L^\infty(\Omega)$, then
%	\[
%	\pnorm[\infty]{u} \le C \pnorm[\infty]{f}^{1/(p-1)}.
%	\]
\end{lemma}

\noindent
Assume that $sp \le N$, $h \in \widetilde{\B}_p$ and $K_i \in  \widetilde{\B}_{q_i}$.
First we show that the critical groups of $\Phi$ at zero depend only on the values of $g(x,t)$ for small $|t|$.

\begin{lemma} \label{Lemma 4.5}
Let $\delta > 0$ and let $\vartheta : \mathbb{R} \to [- \delta,\delta]$ be a smooth nondecreasing function such that $\vartheta(t) = - \delta$ for $t \le - \delta$, $\vartheta(t) = t$ for $- \delta/2 \le t \le \delta/2$ and $\vartheta(t) = \delta$ for $t \ge \delta$. Set
\[
\Phi_1(u) := \frac{1}{p} \|u\|^p - \frac{\lambda}{p}\int_\Omega h(x) |u|^pdx -\int_{\Omega} G(x,\vartheta(u))dx, \quad u \in W_0^{s,p}(\Omega).
\]
If $0$ is an isolated critical point of $\Phi$, then it is also an isolated critical point of $\Phi_1$ and
$$
C^q(\Phi,0) \approx C^q(\Phi_1,0)\quad \forall q.
$$
\end{lemma}

\begin{proof}
By applying \cite[Proposition 4.1]{PS} to the family of functionals
\[
\Phi_\tau(u) := \frac{1}{p} \|u\|^p - \frac{\lambda}{p}\int_\Omega h(x) |u|^pdx-\int_{\Omega}G(x,(1 - \tau)\, u + \tau\, \vartheta(u)) dx,
\quad u \in W_0^{s,p}(\Omega),\, \tau \in [0,1]
\]
in a small ball $B_\varepsilon(0) = \left\{u \in W_0^{s,p}(\Omega): \|u\| \le \varepsilon \right\}$,
observing that each $\Phi_\tau$ satisfies the Palais-Smale condition over $B_\varepsilon(0)$ in light of Lemma \ref{PS1.5} and
that the map $[0,1]\ni\tau \to \Phi_\tau\in C^1(B_\varepsilon(0),\mathbb{R})$
is continuous, we need to show that for sufficiently small $\varepsilon$, $B_\varepsilon(0)$ contains no critical point of any $\Phi_\tau$ other than $0$.
Suppose $u_j \to 0$ in $W_0^{s,p}(\Omega)$, $\Phi_{\tau_j}'(u_j) = 0,\, \tau_j \in [0,1]$ and $u_j \ne 0$. Then
\begin{equation*} \label{perturb2}
\begin{cases}
 (-\Delta)_p^su_j= \lambda h(x)|u_j|^{p-2}u_j + g_j(x,u_j) \quad &\text{in } \Omega,\\
  u_j=0\quad &\text{in } \mathbb{R}^N\setminus \Omega,
  \end{cases}
\end{equation*}
where
\[
g_j(x,t):= (1 - \tau_j + \tau_j\, \vartheta'(t))\, g(x,(1 - \tau_j)\, t + \tau_j\, \vartheta(t)).
\]
Following the proof of \cite[Lemma 4.8]{PS} where we use Lemma \ref{Lemma 4.4} several times, we obtain that
$u_j \in L^\infty(\Omega)$ and $u_j \to 0$ in $L^\infty(\Omega)$ by means of Lemma~\ref{summab}.\
Then, for sufficiently large $j\in {\mathbb N}$, $|u_j(x)| \le \delta/2$ for a.a.\ $x\in\Omega$ and
 hence $\Phi'(u_j) = \Phi_{\tau_j}'(u_j) = 0$, contradicting the assumption that $0$ was an isolated critical point of $\Phi$.
\end{proof}

\noindent
Lemma \ref{Lemma 4.5} and Theorem \ref{Theorem 4.1} immediately give

\begin{theorem}[Critical groups III]
	\label{Theorem 4.6}
Assume that $sp \le N$, $h \in \widetilde{\B}_p$ with $|\{x \in \Omega : h(x) >0\}| >0$, $g$ satisfies
condition \eqref{ggrowth} with $q_i\in(p,p^\ast_s)$ and $K_i \in \widetilde{\B}_{q_i}$,
and $0$ is an isolated critical point of $\Phi$. Then we have
\begin{enumerate}
\item $C^0(\Phi,0) \approx \mathbb{Z}_2$ and $C^q(\Phi,0) = 0$ for $q \ge 1$ if $\lambda = \lambda_1$ and, for some $\delta > 0$, $G(x,t) \le 0$ for a.a. $x \in \Omega$ and $|t| \le \delta$.
\item $C^k(\Phi,0) \ne 0$ in the following cases:
\begin{enumerate}
\item $\lambda = \lambda_k < \lambda_{k+1}$ and, for some $\delta > 0$, $G(x,t) \ge 0$ for a.a. $x \in \Omega$ and $|t| \le \delta$;
\item $\lambda_k < \lambda_{k+1} = \lambda$ and, for some $\delta > 0$, $G(x,t) \le 0$ for a.a. $x \in \Omega$ and $|t| \le \delta$.
\end{enumerate}
\end{enumerate}
\end{theorem}

%######################## SECTION 5. NONTRIVIAL SOLUTIONS ######################%

\section{Nontrivial solutions}

\noindent
We now investigate the existence of nontrivial solutions of the problem
\begin{equation} \label{nontrivial}
\begin{cases}
 (-\Delta)_p^su= \lambda h(x)|u|^{p-2}u +K(x)|u|^{q-2}u+ g(x,u) \quad &\text{in } \Omega,\\
  u=0\quad &\text{in } \mathbb{R}^N\setminus \Omega,
  \end{cases}
\end{equation}
where $q \in (p,p^\ast_s)$, $K \in \B_q$ satisfies
\begin{equation} \label{5.2}
\inf_{x \in \Omega}\, K(x) > 0,
\end{equation}
and $g$ satisfies \eqref{ggrowth} with each $q_i \in (p,q)$. Once again, we will first find suitable subclasses of $\B_p$ and 
$\B_{q_i}$ for the weights $h$ and $K_i$, respectively, to ensure the Palais-Smale condition.

%\begin{definition}[Class of weights ${\mathcal{A}}_t^q$]
%	\label{Definition 5.1}
%For $q \in (1,p^\ast_s)$ and $t \in [1,q)$, let $\mathcal{A}_t^q$ denote the class of
%functions $K$ such that $K \in L^r(\Omega)$ for some $r \in (1,\infty)$ satisfying
%\begin{equation} \label{At}
%\frac{1}{r} +\frac{t}{q} \le 1.
%\end{equation}
%\end{definition}

\begin{definition}[Class of weights ${\B}_t^q$]
	 \label{Definition 5.1}
	For $q \in (1,p^\ast_s)$ and $t \in [1,q)$, let $\B_t^q$ denote the class of
functions $K$ such that $K \rho^{sa} \in L^r(\Omega)$ for some $a \in [0,t - 1]$ and $r \in (1,\infty)$ satisfying
\begin{equation} \label{Bt}
\frac{1}{r} + \frac{a}{p} + \frac{t-a}{q} \le 1.
\end{equation}
\end{definition}

\noindent
Clearly,  $\B_t^q \subset \B_t.$

%\begin{definition}[Unifying class of weights]
%We denote by $\mathcal{C}_t^q$ the class $\mathcal{A}_t^q$ when
%$sp=1$ and the class $\mathcal{B}_t^q$ when $sp\neq 1$.
%\end{definition}

\begin{lemma} \label{Lemma 5.2}
Let $q \in [1,p^\ast_s)$, $t \in [1,q)$ and $K \in \B_t^q$. Then there exist
$0\le m < p$ and, for any $\varepsilon > 0$, a constant $C(\varepsilon) > 0$ such that
\[
\int_\Omega |K(x)|\, |u|^t\, dx \le C(\varepsilon) \|u\|^m + \varepsilon |u|^q_q
\,\,\,\quad \text{for every $u \in W_0^{s,p}(\Omega)$}.
\]
\end{lemma}

\begin{proof}
Denoting $a$ and $r$ as in Definition \ref{Definition 5.1}, by the H\"older inequality, we have
\[
\int_\Omega |K(x)|\, |u|^t\, dx = \int_\Omega |K \rho^{sa}| \left|\frac{u}{\rho^s}\right|^a |u|^{t-a}\, dx \le \left|K \rho^{sa}\right|_r \left|\frac{u}{\rho^s}\right|^a_p \left|u\right|^{t-a}_b,
\]
where $1/r + a/p + (t-a)/b = 1$ and hence $b \le q$ by \eqref{Bt}. 
The last expression is less than or equal to $C \|u\|^{a} |u|^{t - a}_q$ due to $|u|_b \le C |u|_q$ 
and Theorem \ref{Hardy}.
The Young inequality implies that the latter is less than or equal to $C(\varepsilon) \|u\|^{m} + \varepsilon |u|^q_q$, where $a/m + (t-a)/q = 1$ when $a>0$ and $m=0$ when $a=0$. Hence $0\leq m < p$ by means of \eqref{Bt}.
\end{proof}

\noindent
Assuming that $h \in \B_p^q$ and $K_i \in \B_{q_i}^q$,
modifying the argument in the proof of \cite[Lemma 5.3]{PS}, the associated functional
\[
\Phi(u) := \frac{1}{p} \|u\|^p -\frac{\lambda}{p}\int_\Omega h(x)|u|^pdx - \frac{1}{q}\int_\Omega K(x)|u|^qdx-\int_{\Omega} G(x,u)dx,
\ \quad u \in W_0^{s,p}(\Omega),
\]
where $G(x,t):= \int_0^t g(x,\tau)\, d\tau$, satisfies the Palais-Smale condition.

\begin{lemma} \label{Lemma 5.3}
Any sequence $\{u_n\}_{n\in {\mathbb N}} \subset W_0^{s,p}(\Omega)$ such that $\{\Phi(u_n)\}_{n\in{\mathbb N}}$ is
bounded and $\Phi'(u_n) \to 0$ admits a convergent subsequence in $W_0^{s,p}(\Omega)$.
\end{lemma}

\noindent
Concerning the structure of the sublevel sets
$\Phi^\alpha := \left\{u \in W_0^{s,p}(\Omega) : \Phi(u) \le \alpha \right\}$
for $\alpha < 0$ with $|\alpha|$ large, one can follow the proofs of \cite[Lemma 5.4 and 5.5]{PS} and get
%
%\begin{lemma} \label{Lemma 5.4}
%We have
%\begin{enumerate}
%\item $\displaystyle \sup_{u \in X} \left(\Phi'(u)u - \frac{p+q}{2}\, \Phi(u)\right) < \infty$;
%\item $\displaystyle \lim_{t \to + \infty} \Phi(tu) = - \infty \quad \forall u \in W_0^{s,p}(\Omega) \setminus \{0\}$.
%\end{enumerate}
%\end{lemma}

%\begin{proof}
%(1) We have
%\begin{multline} \label{6.10}
%\Phi'(u)u - \frac{p+q}{2}\, \Phi(u) = - \frac{q-p}{2} \int_\Omega \left[\frac{1}{p}\, |\nabla u|^p - \frac{1}{p}\, \lambda\, h(x)\, |u|^p + \frac{1}{q}\, K(x)\, |u|^q\right] dx\\[10pt]
%+ \int_\Omega \left[\frac{p+q}{2}\, G(x,u) - u\, g(x,u)\right] dx.
%\end{multline}
%By $(B1)$ and Lemma \ref{Lemma 6.2}, for every $\varepsilon > 0$, there exists $C(\varepsilon)$ such that
%\begin{gather}
%\left|\int_\Omega \left[\frac{p+q}{2}\, G(x,u) - u\, g(x,u)\right] dx \right| \le C(\varepsilon) \sum_{i=1}^n ||u||^{m_i} + \varepsilon |u|^q_q,\\[10pt]
%\label{6.12} \left|\int_\Omega h(x)\, |u|^p\, dx \right| \le C(\varepsilon) ||u||^m + \varepsilon |u|^q_q,
%\end{gather}
%where each $m_i < p$ and $m < p$. Combining \eqref{6.2} and \eqref{6.10}--\eqref{6.12} gives
%\[
%\Phi'(u)u - \frac{p+q}{2}\, \Phi(u) \le - \frac{1}{2} \left(\frac{q}{p} - 1\right) ||u||^p + C \left(\sum_{i=1}^n ||u||^{m_i} + ||u||^m\right),
%\]
%from which the conclusion follows.

%(2) This follows from \eqref{6.2} and $(B1)$ since $p < q$ and each $q_i < q$.
%\end{proof}

\begin{lemma} \label{Lemma 5.5}
There exists $\alpha < 0$ such that $\Phi^\alpha$ is contractible in itself.
\end{lemma}

%\begin{proof}
%By Lemma \ref{Lemma 6.4} (1), there exists $\alpha < 0$ such that
%\begin{equation} \label{6.13}
%\Phi'(u)u < 0 \quad \forall u \in \Phi^\alpha.
%\end{equation}
%For $u \in W^{s,\, p}_0(\Omega) \setminus \{0\}$, taking into account Lemma \ref{Lemma 6.4} (2), set
%\[
%t(u) = \min\, \left\{t \ge 1 : \Phi(tu) \le \alpha \right\},
%\]
%and note that the function $u \mapsto t(u)$ is continuous by \eqref{6.13}. Then $u \mapsto t(u)\, u$ is a retraction of $W^{s,\, p}_0(\Omega) \setminus \{0\}$ onto $\Phi^\alpha$, and the conclusion follows since $W^{s,\, p}_0(\Omega) \setminus \{0\}$ is contractible in itself.
%\end{proof}

\noindent
We now state  the main existence result of this section.

\begin{theorem}
	\label{nontrivialth}
Assume that $\lambda \ge 0$, $q \in (p,p^\ast_s)$, $h \in \B_p^q$ with $|\{x \in \Omega : h(x) >0\}| >0$, $K \in \B_q$ satisfies \eqref{5.2}, and $g$ satisfies \eqref{ggrowth} with $K_i \in \B_{q_i}^q$ for $i = 1,\dots,n.$  Then problem \eqref{nontrivial} has a nontrivial weak solution in each of the following cases:
\begin{enumerate}
\item $\lambda \notin \left\{\lambda_k : k \ge 1 \right\}$;
\item $G(x,t) \ge 0$ for a.a. $x \in \Omega$ and all $t \in \mathbb{R}$;
\item $G(x,t) \le 0$ for a.a. $x \in \Omega$ and all $t \in \mathbb{R}$;
\item $sp > N$ and, for some $\delta > 0$, $G(x,t) \ge 0$ for a.a. $x \in \Omega$ and $|t| \le \delta$;
\item $sp > N$ and, for some $\delta > 0$, $G(x,t) \le 0$ for a.a. $x \in \Omega$ and $|t| \le \delta$;
\item $sp \le N$, $h \in \widetilde{\B}_p$, $K_i \in \widetilde{\B}_{q_i}$ and, for some $\delta > 0$, $G(x,t) \ge 0$ for a.a. $x \in \Omega$ and $|t| \le \delta$;
\item $sp \le N$, $h \in \widetilde{\B}_p$, $K_i \in \widetilde{\B}_{q_i}$ and, for some $\delta > 0$, $G(x,t) \le 0$ for a.a. $x \in \Omega$ and $|t| \le \delta$.
\end{enumerate}
\end{theorem}

\begin{proof}
Suppose that $0$ is the only critical point of $\Phi$. Taking $U = W_0^{s,p}(\Omega)$ in \eqref{criticalgroup}, we have
\[
C^q(\Phi,0) = H^q(\Phi^0,\Phi^0 \setminus \{0\}).
\]
Let $\alpha < 0$ be as in Lemma \ref{Lemma 5.5}. Since $\Phi$ has no other critical points and satisfies the
Palais-Smale condition by Lemma \ref{Lemma 5.3}, $\Phi^0$ is a deformation retract of $W_0^{s,p}(\Omega)$ and $\Phi^\alpha$ is a deformation retract of $\Phi^0 \setminus \{0\}$ by the second deformation lemma. So
\[
C^q(\Phi,0) \approx H^q(W_0^{s,p}(\Omega),\Phi^\alpha) = 0 \quad \forall q
\]
since $\Phi^\alpha$ is contractible in itself, contradicting Theorem \ref{Theorem 4.1}, Theorem \ref{Theorem 4.2}, or Theorem \ref{Theorem 4.6}.
\end{proof}

%############################MULTIPLICITY############################%

\section{Multiplicity}
\noindent
In this section we show the existence of infinitely many solutions via the Fountain Theorem. Since $W_0^{s,p}(\Omega)$ is separable,
there exist $\{e_n\}_{n\in {\mathbb N}}\subset W_0^{s,p}(\Omega)$  and $\{f_n\}_{n\in {\mathbb N}} \subset W_0^{-s,p'}(\Omega)$ with
$$
W_0^{s,p}(\Omega) = \overline{\mbox{span}\{e_n\}_{n=1}^{\infty}},\qquad
W_0^{-s,p'}(\Omega) = \overline{\mbox{span}\{f_n\}_{n=1}^{\infty}},\qquad
\langle f_i, e_j \rangle =
\begin{cases}
1 & \text{ if } i = j, \\
0 & \text{ if } i \ne j,
\end{cases}
$$
where $\langle \cdot,\cdot\rangle$ is the duality pairing between
$W_0^{-s,p'}(\Omega)$ and $W_0^{s,p}(\Omega)$ (see \cite[Section 17]{Zhao}). Let
$$
X_n = \mbox{span}\{e_n\}, \quad\,\,
Y_n = \bigoplus_{k=1}^n X_k, \quad\,\,
Z_n = \overline{\bigoplus_{k=n}^\infty X_k}.
$$
With $X_n, Y_n, Z_n$ taken as the above, we have \cite[Fountain Theorem]{Willem}

%--------------------PROPOSITION: FOUNTAIN THEOREM--------------- %
\begin{theorem}
	\label{Fountain_thm}
	Let $\Phi \in C^1(W_0^{s,p}(\Omega),\mathbb{R})$ be
	even and suppose that there exist $\rho_n >\gamma_n>0$ such that
	\begin{itemize}
		\item[$(\H_1)$] $a_n=\inf\limits_{u \in Z_n,\, \|u\|=\gamma_n}\Phi(u) \to +\infty$ as $n \to \infty$;
		\item[$(\H_2)$] $b_n=\max\limits_{u \in Y_n,\, \|u\|=\rho_n}\Phi(u) \leq 0$;
		\item[$(\H_3)$] $\Phi$ satisfies the Palais-Smale condition at the level $c$ for all $c>0$.
	\end{itemize}
	Then $\Phi$ has a sequence of critical values tending to $+\infty$.
\end{theorem}

\noindent
Invoking this theorem, we obtain the existence of infinitely many solutions for problem \eqref{nontrivial}.
\begin{theorem}
	\label{multip}
	Assume that $q \in (p,p^\ast_s)$, $h \in \B_p^q$ with $|\{x \in \Omega : h(x) >0\}| >0,$ $K \in \B_q$ satisfies \eqref{5.2}, $g$ satisfies \eqref{ggrowth} with $K_i \in \B_{q_i}^q$ and $g(x,-u)=-g(x,u).$ Then problem \eqref{nontrivial} admits a sequence $\{u_n\}_{n\in {\mathbb N}}\subset W_0^{s,p}(\Omega)$ of solutions such that $\Phi(u_n)\to +\infty$ as $n\to\infty$.
\end{theorem}

\begin{proof}
Due to Lemma~\ref{Lemma 5.3}, we only need to verify conditions $(\H_1)$ and $(\H_2)$.
For $(\H_1)$, let
$$
\beta_n':=\sup \{|u|_q: u\in Z_n, \|u\|=1\},
\qquad\beta_n'':=\sup \{|u|_b: u\in Z_n, \|u\|=1\},
$$
where $b$ satisfies $1/r+a/p+(q-a)/b=1,$
where $a,r$ are as in $\B_q$. Set $\beta_n=\max\{\beta_n',\beta_n''\}$. Due to the compact injections $W_0^{s,p}(\Omega)\hookrightarrow\hookrightarrow L^b(\Omega)$ and $W_0^{s,p}(\Omega)\hookrightarrow\hookrightarrow L^q(\Omega)$ we have that $\beta_n\to 0$ as $n\to \infty$ in view of the abstract result \cite[Lemma 3.3]{Fan3}.  Then there exists
$n_0\in {\mathbb  N}$ such that $\beta_n\leq 1$ for all $n\geq n_0$.
For each $n\in \mathbb{N},$ define $\gamma_n$ by $\gamma_n :=1/\sqrt{\beta_n}$
%$$ \gamma_n :=
%\begin{cases}
%\frac{1}{\sqrt{\beta_n}}, & \text{ if } n_0\leq n, \\
%1, & \text{ if } 1\leq n <n_0.
%\end{cases}$$
%Obviously, $\gamma_n\geq 1$ for all $n$ and $\gamma_n \to \infty$ as $n\to \infty.$
and hence, $\gamma_n\geq 1$ for all $n\geq n_0$ and $\gamma_n \to \infty$ as $n\to \infty.$
For $u\in Z_n$ with $\|u\|=\gamma_n$, we have
\begin{equation}\label{Est.1}
\Phi( u) \geq \frac{1}{p}\|u\|^p-\frac{|\lambda|}{p}\int_{\Omega}|h||u|^pdx-\frac{1}{q}\int_{\Omega}K|u|^qdx-\sum_{i=1}^{n}\frac{1}{q_i}\int_{\Omega}|K_i||u|^{q_i}dx.
\end{equation}
By Lemma~\ref{Lemma 5.2} for any $\epsilon>0$ there exist $0\le m_i<p$
and $C_i(\epsilon)>0$,  $i=0,\cdots,n,$ such that
$$
\int_{\Omega}|h||u|^pdx\leq \epsilon \|u\|^{m_0}+C_0(\epsilon)|u|^q_q,
\qquad
\int_{\Omega}|K_i||u|^{q_i}dx\leq \epsilon \|u\|^{m_i}+C_i(\epsilon)|u|^q_q,
$$
and, see the proof of Lemma~\ref{C},
$$
\int_{\Omega}K|u|^qdx\leq C \|u\|^{a}|u|_b^{q-a}.
$$
Applying these estimates with $\epsilon<\frac{1}{2(n+1)p}$ we deduce from \eqref{Est.1}
that, for any $n\geq n_0,$
\begin{align*}
\Phi( u) &\geq \frac{1}{p}\|u\|^p-\frac{C}{q}\|u\|^{a}|u|_b^{q-a}-\epsilon\sum_{i=0}^{n}\|u\|^{m_i}-\Big(\sum_{i=0}^{n}C_i(\epsilon)\Big)|u|_q^q\\
&\geq\frac{1}{p}\|u\|^p-\frac{C}{q}\|u\|^{a}|u|_b^{q-a}-(n+1)\epsilon\|u\|^p-\Big(\sum_{i=0}^{n}C_i(\epsilon)\Big)|u|_q^q\\
&\geq\frac{1}{2p}\|u\|^p-\frac{C}{q}\|u\|^{a}|u|_b^{q-a}-\Big(\sum_{i=0}^{n}C_i(\epsilon)\Big)|u|_q^q\\
&\geq \frac{\beta_n^{-\frac{p}{2}}}{2p}-\frac{C}{q}\beta_n^{\frac{q}{2}-a}-\Big(\sum_{i=0}^{n}C_i(\epsilon)\Big)\beta_n^{\frac{q}{2}}\to\infty,
%&=\beta_n^{-\frac{p}{2}}\Big[\frac{1}{2p}-\frac{C}{q}\beta_n^{\frac{p+q}{2}-a}-\Big(\sum_{i=0}^{n}C_i(\epsilon)\Big)\beta_n^{\frac{p+q}{2}}\Big]\to \infty
\end{align*}
as $n\to\infty$ since $0\leq a<\frac{p+q}{2}$ yielding $(\H_1)$.
For $(\H_2)$, again by Lemma~\ref{Lemma 5.2}, for $u\in Y_n$ we have
\begin{align*}
\Phi( u) &\leq \frac{1}{p}\|u\|^p+\frac{|\lambda|}{p}\int_{\Omega}|h||u|^pdx-\frac{1}{q}\int_{\Omega}K|u|^qdx+\sum_{i=1}^{n}\frac{1}{q_i}\int_{\Omega}|K_i||u|^{q_i}dx\\
&\leq\frac{1}{p}\|u\|^p+\frac{|\lambda|}{q}\left(\widehat{C}_0(\epsilon)\|u\|^{m_0}+\epsilon|u|_q^q\right)
+\sum_{i=1}^{n}\frac{1}{q_i}\left(\widehat{C}_i(\epsilon)\|u\|^{m_i}+\epsilon|u|_q^q\right) -\frac{1}{q}|u|_{L^q(K,\Omega)}^q.
\end{align*}
Since ${\rm dim}(Y_{n})<\infty$, the norms $\|\cdot\|$, $|\cdot|_q$ and $|\cdot|_{L^{q}(K,\Omega)} $ are equivalent. As
$p, m_i<q$, choosing $\epsilon$ small enough the last estimate yields
$\Phi(u)\leq 0$ for all $u\in Y_n$ with $\|u\|$ large enough. This completes $(\H_2)$. The proof is now complete.
\end{proof}

%\medskip

%######################## REFERENCES ############################ %

\bigskip


\begin{thebibliography}{99}

\bibitem{A}
D.\ Applebaum,
L\'evy processes, from probability fo finance and quantum groups,
{\em Notices Amer. Math. Soc.} {\bf 51} (2004), 1336--1347.

\bibitem{BF}
L.\ Brasco, G.\ Franzina,
Convexity properties of Dirichlet integrals and Picone-type inequalities,
{\em Kodai Math. J.} {\bf 37} (2014), 769--799.

\bibitem{BP}
L.\ Brasco, E.\ Parini, The second eigenvalue of the fractional $p$-Laplacian, {\em Adv. Calc. Var.}, to appear
\href{http://www.degruyter.com/view/j/acv.ahead-of-print/acv-2015-0007/acv-2015-0007.xml?format=INT}
{DOI: 10.1515/acv-2015-0007}.

\bibitem{C}
L.A.\ Caffarelli,
Nonlocal equations, drifts and games,
Non.\ Partial Diff. Eq., Abel Symposia {\bf 7} (2012), 37--52.

%\bibitem{CS}
%L.A.\ Caffarelli, L.\ Silvestre,
%An extension problem related to the fractional Laplacian,
%{\em Comm. Partial Differential Equations} {\bf 32} (2007) 1245--1260.

\bibitem{CS1}
X.\ Cabr\'e, Y.\ Sire,
Nonlinear equations for fractional Laplacians I: regularity, maximum principles, and Hamiltonian estimates,
{\em Ann. Inst. Henri Poincar\'e Nonlinear Analysis} {\bf 31} (2014), 23--53.

\bibitem{Dyda}
B.\ Dyda,
A fractional order Hardy inequality,
{\em Illinois J. Math.} {\bf 48} (2004), 575--588.

%\bibitem{Franzina} G. Franzina, G. Palatucci,
%Fractional $p$-eigenvalues,
%{\em Riv. Mat. Univ. Parma} {\bf 5} (2014), 315–-328.

%\bibitem{Grisvard} P. Grisvard,
%Elliptic problems in nonsmooth domains,
%Monog. Studies in Math., {\bf 24} Pitman, Boston,  1985.

%\bibitem{Ihnatsyeva} L. Ihnatsyeva, J. Lehrb\"ack, H. Tuominen, A.V. V\"ah\"akangas
%\emph{Fractional Hardy inequalities and visibility of the boundary}, Studia Math. 224 (2014), 47-80.

\bibitem{otani1}
S.\ Hashimoto, M.\ Otani,
Elliptic equations with singularity on the boundary,
{\em Differential Integral Equations} {\bf 12} (1999), 339--349.

\bibitem{otani2}
S.\ Hashimoto, M.\ Otani,
Sublinear elliptic equations and eigenvalue problems with singular coefficients,
{\em Commun. Appl. Anal.} {\bf 6} (2002), 535--547.

\bibitem{Ladyzhenskaya}
O. Ladyzhenskaya, N. Uraltseva,
Linear and quasilinear elliptic equations, Academic Press, 1968.

%\bibitem{MR1422006}
%K.~C. Chang and N.~Ghoussoub.
%\newblock The {C}onley index and the critical groups via an extension of
%  {G}romoll-{M}eyer theory.
%\newblock {\em Topol. Methods Nonlinear Anal.}, 7(1):77--93, 1996.

%\bibitem{MR1926378}
%J.-N. Corvellec and A.~Hantoute.
%\newblock Homotopical stability of isolated critical points of continuous
%  functionals.
%\newblock {\em Set-Valued Anal.}, 10(2-3):143--164, 2002.
%\newblock Calculus of variations, nonsmooth analysis and related topics.

%\bibitem{MR2661274}
%Marco Degiovanni, Sergio Lancelotti, and Kanishka Perera.
%\newblock Nontrivial solutions of {$p$}-superlinear {$p$}-{L}aplacian problems
%  via a cohomological local splitting.
%\newblock {\em Commun. Contemp. Math.}, 12(3):475--486, 2010.

\bibitem{MR57:17677}
E.R.\ Fadell, P.H. Rabinowitz,
Generalized cohomological index theories for {L}ie group actions with
  an application to bifurcation questions for {H}amiltonian systems,
{\em Invent. Math.} {\bf 45} (1978), 139--174.

\bibitem{Fan3}
X. Fan, X. Han,
Existence and multiplicity of solutions for $p(x)$-Laplacian equations in $\mathbb{R}^N$,
{\em Nonlinear Anal.} {\bf 59} (2004), 173--188.

%\bibitem{FSV}
%A.\ Fiscella, R.\ Servadei, E.\ Valdinoci,
%Density properties for fractional Sobolev spaces,
%{\em Ann. Acad. Sci. Fenn. Math.}, {\bf 40} (2015), 235--253.

\bibitem{Franzina}
G.\ Franzina, G.\ Palatucci,  Fractional $p$-eigenvalues, {\em Riv. Mat. Univ. Parma} {\bf 5} (2014), 315--328.

\bibitem{kasm}
A.\ Iannizzotto, S.\ Liu, K.\ Perera, M.\ Squassina,
Existence results for fractional $p$-Laplacian problems via Morse theory,
{\em Adv. Calc. Var.}, to appear
\href{http://dx.doi.org/10.1515/acv-2014-0024}{DOI: 10.1515/acv-2014-0024}.

\bibitem{IMS1}
A.\ Iannizzotto, S.\ Mosconi, M.\ Squassina,
Global H\"older regularity for the fractional $p$-Laplacian,
{\em Rev. Mat. Iberoam.}, to appear \href{http://arxiv.org/abs/1411.2956}{arxiv.org/abs/1411.2956}.

\bibitem{IS}
A.\ Iannizzotto, M.\ Squassina,
Weyl-type laws for fractional $p$-eigenvalue problems,
{\em Asymptotic Anal.} {\bf 88} (2014), 233--245.

\bibitem{hardyy}
L.\ Ihnatsyeva, J.\ Lehrb\"ack, H.\ Tuominen, A.\ V\"ah\"akangas, 
Fractional Hardy inequalities and visibility of the boundary,
{\em Studia Math.} {\bf 224} (2014), 47--80. 

%\bibitem{IS1}
%A.\ Iannizzotto, M.\ Squassina,
%$1/2$-Laplacian problems with exponential nonlinearity,
%{\em J. Math. Anal. Appl.} {\bf 414} (2014) 372--385.

%\bibitem{MR1009077}
%Mohammed Guedda and Laurent V{\'e}ron.
%\newblock Quasilinear elliptic equations involving critical {S}obolev
%  exponents.
%\newblock {\em Nonlinear Anal.}, 13(8):879--902, 1989.

%\bibitem{MR96a:58045}
%Shu~Jie Li and Michel Willem.
%\newblock Applications of local linking to critical point theory.
%\newblock {\em J. Math. Anal. Appl.}, 189(1):6--32, 1995.

\bibitem{erk-lind}
E.\ Lindgren, P.\ Lindqvist,
Fractional eigenvalues,
{\em Calc. Var. PDE} {\bf 49} (2014), 795--826.

\bibitem{MRmpsy}
S.\ Mosconi, K.\ Perera, M.\ Squassina, Y. Yang,
The Brezis-Nirenberg Problem for the fractional $p$-Laplacian,
preprint  \href{http://arxiv.org/abs/1508.00700}{arxiv.org/abs/1508.00700}.

%\bibitem{MR0163054}
%J.\ Ne{\v{c}}as, Sur une m\'ethode pour r\'esoudre les \'equations aux d\'eriv\'ees
%partielles du type elliptique, voisine de la variationnelle.
%{\em Ann. Scuola Norm. Sup. Pisa} {\bf 16} (1962), 305--326.

%\bibitem{MR1700283}
%Kanishka Perera.
%\newblock Homological local linking.
%\newblock {\em Abstr. Appl. Anal.}, 3(1-2):181--189, 1998.

\bibitem{P}
K.\ Perera, Nontrivial critical groups in $p$-Laplacian problems via the Yang index,
{\em Topol. Methods Nonlinear Anal.} {\bf 21} (2003), 301--309.

%\bibitem{MR1998432}
%Kanishka Perera.
%\newblock Nontrivial critical groups in {$p$}-{L}aplacian problems via the
%  {Y}ang index.
%\newblock {\em Topol. Methods Nonlinear Anal.}, 21(2):301--309, 2003.

\bibitem{MR2640827}
K.\ Perera, R.P. Agarwal,  D.\ O'Regan,
Morse theoretic aspects of {$p$}-{L}aplacian type operators,
{\bf 161} {\em Mathematical Surveys and Monographs}.
Amer.\ Math.\ Soc., Providence, RI, 2010.

\bibitem{PS}
K.\ Perera, I.\ Sim,
$p$-Laplace equations with singular weights,
{\em Nonlinear Anal.} {\bf 99} (2014), 167--176.

\bibitem{RS}
X.\ Ros-Oton, J.\ Serra,
The Dirichlet problem for the fractional Laplacian: regularity up to the boundary,
{\em J. Math. Pures Appl.} {\bf 101} (2014), 275--302.

\bibitem{RS1}
X.\ Ros-Oton, J.\ Serra,
The Poho\v zaev identity for the fractional laplacian,
{\em Arch. Rat. Mech. Anal.} {\bf 213} (2014), 587--628.

%\bibitem{RS2}
%X.\ Ros-Oton, J.\ Serra,
%\newblock Nonexistence results for nonlocal equations with critical and supercritical nonlinearities,
%\newblock {\em Comm. Partial Differential Equations}, to appear.

\bibitem{Ry}
R.\ Kajikiya, Superlinear elliptic equations with singular coefficients on the boundary,
{\em Nonlinear Anal.} {\bf 73} (2010), 2117--2131.

%\bibitem{MR3160532}
%Kanishka Perera and Inbo Sim
%\newblock $p$-Laplace equations with singular weights.
%\newblock {\em Nonlinear Anal.}, 99(2014):167–17, 2014.

\bibitem{SV}
R.\ Servadei, E.\ Valdinoci,
Mountain pass solutions for non-local elliptic operators,
{\em J. Math. Anal. Appl.} {\bf 389} (2012), 887--898.

%\bibitem{SV1}
%R.\ Servadei, E.\ Valdinoci,
%\newblock Variational methods for non-local operators of elliptic type,
%\newblock {\em Discrete Contin. Dyn. Syst.} {\bf 33} (2013) 2105--2137.
%
%\bibitem{SV2}
%R.\ Servadei, E.\ Valdinoci,
%\newblock Lewy-Stampacchia type estimates for variational inequalities driven by (non)local operators,
%\newblock {\em Rev. Mat. Iberoam.} {\bf 29} (2013) 1091--1126.

\bibitem{SV3}
R.\ Servadei, E.\ Valdinoci,
The Brezis-Nirenberg result for the fractional Laplacian,
{\em Trans. Amer. Math. Soc.} {\bf 367} (2015), 67--102.

%\bibitem{SV4}
%R.\ Servadei, E.\ Valdinoci,
%\newblock Weak and viscosity solutions of the fractional Laplace equation,
%\newblock {\em Publ. Mat.}, {\bf 58} (2014) 133--154.
\bibitem{usami}
H.\ Usami,
On a singular elliptic boundary value problem in a ball,
{\em Nonlinear Anal.} {\bf 13} (1989), 1163--1170.

\bibitem{Willem}
M. Willem,
Minimax theorems, Birkh\"{a}user, Boston, 1996.

\bibitem{Zhao}
J.F. Zhao,
Structure Theory of Banach Spaces, Wuhan University Press, Wuhan, 1991.

\end{thebibliography}
\end{document}